\documentclass{article}
\pdfoutput=1

\usepackage[ascii]{inputenc}
\usepackage[USenglish]{babel}
\usepackage{amsmath}
\usepackage{amsfonts}
\usepackage{amsthm}
\usepackage{paralist}
\usepackage{hyperref}

\title{Towards Computing Vector Partition Functions by Iterated Partial Fraction Decomposition}

\author{Thomas Bliem\thanks{Supported by the Deutsche Forschungsgemeinschaft.}}

\date{}

\hypersetup{bookmarksnumbered,pdfauthor={Thomas Bliem},pdftitle={Towards computing vector partition functions by iterated partial fraction decomposition}}

\newcommand{\arxiv}[1]{\href{http://arxiv.org/abs/#1}{\nolinkurl{arXiv:#1}}}

\newcommand{\urn}[1]{\href{http://nbn-resolving.de/urn:#1}{\nolinkurl{urn:#1}}}
\newcommand{\doi}[1]{\href{http://dx.doi.org/#1}{\nolinkurl{doi:#1}}}
\newcommand{\zbl}[1]{\href{http://www.zentralblatt-math.org/zmath/en/advanced/?q=an:#1&format=complete}{Zbl \nolinkurl{#1}}}

\newtheorem{proposition}{Proposition}
\newtheorem{lemma}{Lemma}

\theoremstyle{definition}
\newtheorem{definition}{Definition}
\newtheorem{example}{Example}

\newcommand{\nonnegint}{\mathbf{N}_0} 
\newcommand{\posint}{\mathbf{N}}
\newcommand{\nonnegreal}{\mathbf{R}_+} 
\newcommand{\units}{\times} 
\newcommand{\infrt}[1]{\!\!\sqrt[\infty]{#1}}
\DeclareMathOperator{\const}{const}

\renewcommand{\phi}{\varphi}

\begin{document}

\maketitle

\begin{abstract}
\noindent We investigate the possibilities to calculate vector partition functions by means of iterated partial fraction decomposition, as suggested by Beck (2004).
Particularly, for an important type of families of rational functions, we describe an algorithm to compute the numerators in their partial fraction decomposition ``formally,'' i.e., as a formal expression in the parameter.
We also analyze the type of generalized rational functions that appear during the execution of an algorithm based on iterated partial fraction decomposition and explain how to handle these objects.
\end{abstract}

\section{Introduction}
\label{nzBY63K7}

We denote the set of nonnegative integers by $\nonnegint$ and the set of nonnegative real numbers by $\nonnegreal$.
Let $c_1, \ldots, c_d \in \mathbf{Z}^m \subset \mathbf{R}^m$, contained in an open half-space.
The typical situation to have in mind is that $d > m$, so that the family $(c_i)$ is linearly dependent.
To this data, associate the \emph{vector partition function} $\phi_{c_1, \ldots, c_d} : \mathbf{Z}^m \to \nonnegint$ by
\[
  \phi_{c_1, \ldots, c_d}(b) = \lvert \{ x \in \nonnegint^d : x_1 c_1 + \cdots + x_d c_d = b \} \rvert .
\]
So $\phi_{c_1, \ldots, c_d}(v)$ is the number of decompositions of $b$ as a nonnegative integral linear combination of $c_1, \ldots, c_d$.
In a more geometric language, $\phi_{c_1, \ldots, c_d}(b)$ is the number of integral points in the parametric polytope
\[
	P_{c_1, \ldots, c_d}(b) = \{ x \in \nonnegreal^d : x_1 c_1 + \cdots + x_d c_d = b \} .
\]
Thus the theory of vector partition functions generalizes Ehrhart theory \cite{ehrhart1962}, the latter corresponding to the study of the restriction of $\phi_{c_1, \ldots, c_d}$ to a ray.

Vector partition functions occur throughout mathematics as a convenient way to encode polyhedral models, such as:
\begin{inparaenum}
\item In representation theory, as Kostant's partition function \cite{kostant1959}.
\item Also in representation theory, as the vector partition functions constructed by Billey et al.\ \cite{billey2004} from Gelfand-Tsetlin patterns and by myself \cite{bliem2010} from Littelmann patterns and Berenstein-Zelevinsky polytopes.
\item For counting integer flows in networks, as discussed by Baldoni et al.\ \cite{baldoni2004}.
\item In the optimization of computer programs (see \cite{verdoolaege2007} for an extensive list of applications in this area).
\end{inparaenum}

The usefulness of expressing a given function by a formula involving vector partition functions partially relies on the following theorem, proved by Blackley \cite{blakley1964}, Dahmen and Micchelli \cite{dahmen1988}, and Sturmfels \cite{sturmfels1995}:
Given $c_1, \ldots, c_d$, there is a fan $F$ in $\mathbf{R}^d$ such that the restriction of $\phi_{c_1, \ldots, c_d}$ to each maximal cone $C \in F$ is given by a quasi-polynomial $f_C$, and $\phi_{c_1, \ldots, c_d}$ vanishes outside $F$.
Suppose you are given a formula for some interesting quantity involving a vector partition function.
If you succeed in determining explicitly its fan $F$ and its quasi-polynomials $(f_C)$, you obtain a very explicit general formula for the original quantity.
For this approach to work, it is of course essential to have effective means to determine $F$ and $(f_C)$.
To date, there are a number of algorithms known allowing to do this, employing a variety of ideas.
Some of the approaches I am aware of are the following:
\begin{enumerate}
\item Using Barvinok's algorithm \cite{barvinok1994}.
This has been worked out by Verdoolaege et al.\ \cite{verdoolaege2007}.
\item Using the Szenes-Vergne residue formula \cite{szenes2003}. This was employed by Baldoni et al.\ \cite{baldoni2006} for Kostant's partition function associated with classical root systems and by myself \cite{bliem2010} for vector partitions functions obtained from generalized Gelfand-Tsetlin patterns.
\item Using Paradan's wall crossing formula which relates the quasi-polynomials of adjacent maximal cones of $F$, see Boysal and Vergne \cite{boysal2008}.
\item Iteratively rewriting the generating function as a sum of rational functions with a special pole structure.
Milev is working on this approach \cite{milev2009}.
Note that he uses the term ``partial fraction decomposition'' in a different sense than we do.
\end{enumerate}
Beck suggested a further algorithm \cite{beck2004}, based on iterated partial fraction decomposition of the generating function.
The aim of this extended abstract is to report on work in progress to work this idea out.

\section{Preliminaries}
\label{CpwJ444b}

This section is a summary of Beck's suggestions \cite{beck2004} to compute vector partition functions.
We denote the set of positive integers by $\mathbf{N}$.
Let $m, d \in \posint$ and let $A \in \mathbf{Z}^{(m,n)}$ be an $(m \times d)$-matrix with integral coefficients.
We consider $A$ as a linear map $\mathbf{R}^d \to \mathbf{R}^m$.
Suppose that $\mathrm{ker}(A) \cap \nonnegreal^d = \{ 0 \}$.
Then we can associate with $A$ the \emph{vector partition function} $\phi_A : \mathbf{Z}^m \to \nonnegint$ by
\[
	\phi_A(b) = \lvert \{ x \in \nonnegint^d : Ax = b \} \rvert .
\]
Let $c_1, \ldots, c_d$ be the column vectors of $A$.
Then $\phi_A$ as defined above coincides with $\phi_{c_1, \ldots, c_d}$ as defined in \autoref{nzBY63K7}.
We use the standard multiexponent notation, $z^b = z_1^{b_1} \cdots z_m^{b_m}$.
The generating function of $\phi_A$,
\[
	f_A(z) = \sum_{b \in \mathbf{Z}^m} \phi_A(b)z^b ,
\]
converges on $\{ z \in \mathbf{C}^m : \lvert z^{c_k} \rvert < 1$ for $k = 1, \ldots, d \}$ to the rational function
\begin{equation}
\label{Jc8LhJrg}
	f_A(z) = \prod_{k=1}^d \frac{1}{1-z^{c_k}} .
\end{equation}
By definition of the generating function, $\phi_A(b)$ is the $z^b$-coefficient of $f_A(z)$.
Equivalently, it is the constant coefficient of $f_A(z)z^{-b}$,
\[
	\phi_A(b) = \operatorname{const} f_A(z)z^{-b} .
\]
Consider $f_A(z)z^{-b}$ as a rational function in one variable $z_m$ over the field $K_{m-1} = \mathbf{C}(z_1, \ldots, z_{m-1})$.
Assume there is a well-defined notion of the constant coefficient of a rational function.
Denote the constant coefficient of $f_A(z)z^{-b} \in K_{m-1}(z_m)$ by $\operatorname{const}_{z_n} f_A(z)z^{-b} \in K_{m-1}$.
Then by iterating this procedure we get the expression
\[
	\phi_A(b) = \operatorname{const}_{z_1} \cdots \; \operatorname{const}_{z_m} f_A(z)z^{-b} \in K_0 = \mathbf{C}
\]
for the vector partition function $\phi_A$.
Hence, the problem of computing $\phi_A$ is reduced to the problem of finding the constant coefficient of a given rational function in one variable.
This can be done using partial fraction decomposition, as described in the following.

Let $K$ be a field.
Consider the field $K(w)$ of rational functions in one variable $w$ over $K$.
(In the setting of the previous paragraph, $K = \mathbf{C}(z_1, \ldots, z_{m-1})$ and $w = z_m$.)
Let $f(w) = \frac{r(w)}{s(w)} \in K$ for polynomials $r(w), s(w) \in K[w]$.
Suppose $\mathrm{deg}(r(w)) < \mathrm{deg}(s(w))$.
We want to compute $\operatorname{const}_w f(w)w^{-b}$ as a function of $b \in \mathbf{Z}$.
We can suppose $s(0) \neq 0$ by a shift in $b$.
Fix a decomposition $s(w) = s_1(w) \cdots s_d(w)$ into pairwise coprime factors.
Then $s_k(w)$ and $w$ are also coprime for $k = 1, \ldots, d$.
Hence, for all $b \geq 0$, we get a partial fraction decomposition
\begin{equation} \label{Nb7ZAmCp}
	f(w)w^{-b}
	= \frac{A_1(b; w)}{s_1(w)}
	+ \cdots
	+ \frac{A_d(b; w)}{s_d(w)}
	+ \frac{B(b; w)}{w^b}
\end{equation}
for some $A_k(b; w), B(b; w) \in K[w]$ such that $\mathrm{deg}(A_k(b; w)) < \mathrm{deg}(s_k(w))$ and $\mathrm{deg}(B(b; w)) < b$.
Then
\[ \begin{split}
	\operatorname{const}_w f(w)w^{-b}
	&= \operatorname{const}_w \frac{A_1(b; w)}{s_1(w)}
	+ \cdots
	+ \operatorname{const}_w \frac{A_d(b; w)}{s_d(w)}
	+ \operatorname{const}_w \frac{B(b; w)}{w^b} \\
	&= \frac{A_1(b; 0)}{s_1(0)}
	+ \cdots
	+ \frac{A_d(b; 0)}{s_d(0)} .
\end{split} \]
Note that, in particular, the constant coefficient of $f(w)w^{-b}$ does not depend on $B(b; w)$ at all, and neither on the higher coefficients of $A_k(b;w)$.
Hence we do not need to know the complete partial fraction decomposition of $f(w)w^{-b}$ to proceed with our computations.

By the above ideas it is indeed possible to compute quasi-polynomial expressions for vector partition functions.
Beck demonstrates this by explicitly calculating the case $A = \bigl( \begin{smallmatrix} 1 & 2 & 1 & 0 \\ 1 & 1 & 0 & 1 \end{smallmatrix} \bigr)$.
However, this example is special in some respect:
\begin{inparaenum}
\item The last row of $A$ only contains the numbers $0$ and $1$.
Hence, all roots of the polynomials $1-z^{c_k} \in \mathbf{C}(z_1)[z_2]$ will be elements of $\mathbf{C}(z_1)$.
In general, they would be elements of an algebraic extension.
\item All components of $A$ are nonnegative.
Given a matrix $A$ such that $\phi_A$ is defined, one can always suppose that all components of $A$ are nonnegative by a unimodular change of coordinates, so this assumption is without loss of generality.
Still it is possible to define $\phi_A$ for matrices $A$ with negative components, so it would be nice to understand how to deal with those directly.
\item The columns of $A$ are pairwise linearly independent.
This translates to the fact that the generating function has no multiple poles when considered as a function in $z_2$.
Additional complications will arise if there are multiple poles.
\end{inparaenum}

There is further study needed to make the proposed procedure applicable mechanically.
In particular, the following questions must  be addressed:
\begin{enumerate}
\item Which decomposition of $s(w)$ should one use?
There is a natural one coming from \eqref{Jc8LhJrg}, and the usual fine decomposition into prime powers.
\item What is the constant term of a rational function?
In general this depends on the region of convergence considered.
If handled without the appropriate care, this leads to paradoxical formulas such as
\[ \begin{split}
	1
	&= \operatorname{const}\bigl( 1 + z + z^2 + \cdots \bigr)
	= \operatorname{const} \frac{1}{1-z} \\
	&= \operatorname{const} -z^{-1} \frac{1}{1-z^{-1}}
	= \operatorname{const}\bigl( -z^{-1} - z^{-2} - \cdots \bigr)
	= 0 .
\end{split} \]
\item Which algorithm should be used to compute the partial fraction decomposition?
Note that it is crucial to be able not only to compute a particular partial fraction decomposition of $f(w)w^{-b}$ for a a particular $b \in \nonnegint$, but as a function of the variable $b$.
\item How do we deal with multiple poles of $f_A(w)$?
These occur if and only if $A$ has a pair of linearly dependent columns.
Hence they become increasingly unlikely for higher $m$, but cannot be excluded.
\item If we answer the first question by using the fine decomposition of $s(w)$, lots of summands and roots of unity will appear.
Do we have to compute all of them individually, or can we use the Galois invariance, i.e., the invariance of $\phi_A(b)$ under the $G(\mathbf{Q}^\mathrm{ab} | \mathbf{Q})$-operation, to facilitate this?
\end{enumerate}
We will partially address these questions in the following.
The remaining issues are delayed to a forthcoming full version of this extended abstract, where the word ``Towards'' will be hopefully dropped from the title.

\section{A lemma about partial fraction decomposition}
\label{EE8BVYaT}

By partial fraction decomposition I mean the following:

\begin{definition}
\label{wL3ry9dn}
Let $\mathcal{O}$ be a domain.
Let $\frac rs$ be an element of the quotient field of $\mathcal{O}$.
Fix a decomposition $s = s_1 \cdots s_d$ of $s$ such that $(s_j, s_k) = \mathcal{O}$ for all $j \neq k$.
Then a \emph{partial fraction decomposition} of $\frac rs$ with respect to the decomposition $s_1, \ldots, s_d$ is an expression
\begin{equation}
	\label{9tzyHkce}
	\frac rs = \frac{r_1}{s_1} + \cdots + \frac{r_d}{s_d} + p .
\end{equation}
with $r_1, \ldots, r_d, p \in \mathcal{O}$.
\end{definition}

Note that in this generality the partial fraction decomposition is not unique.
Of course, in the important case when $\mathcal{O}$ is Euclidean, one can standardize further to obtain uniqueness.
Partial fraction decompositions always exist.
This follows from the following lemma, which gives more importantly a way to determine the numerators $r_k$ explicitly.
Note that, once the numerators are determined, the integral part $p$ can be obtained by subtracting the fractional parts from $\frac rs$.

\begin{lemma}
\label{GrJNo6wT}
Let $r, s, s_1, \ldots, s_d \in \mathcal{O}$ as in \autoref{wL3ry9dn}.
Let $c_k = \frac{s}{s_k}$ for $k \in \{ 1, \ldots, n \}$.
Then $c_k$ is invertible in $\mathcal{O}/(s_k)$.
Let $d_k \in \mathcal{O}$ be an inverse of $c_k$ modulo $(s_k)$ and $r_k = rd_k$.
Then the $r_k$ give rise to a partial fraction decomposition of $\frac rs$, i.e., \eqref{9tzyHkce} holds for some $p \in \mathcal{O}$.
\end{lemma}

\begin{proof}
Let $k \in \{1, \ldots, n\}$.
As the elements $s_j$ and $s_k$ were supposed to be coprime for $j \neq k$, these $s_j$ are invertible modulo $(s_k)$.
Hence their product $c_k$ is also invertible modulo $(s_k)$, and we can define $d_k$ and $r_k$ as in the announcement.
By definition, $d_kc_k \equiv 1 \mod (s_k)$.
Hence
\begin{equation}
\label{tQ2cJ9UL}
r_kc_k = rd_kc_k \equiv r \mod (s_k) .
\end{equation}
For $j \neq k$ we have $s_k \mid c_j$, so
\begin{equation}
\label{72DGMp8S}
r_jc_j \equiv 0 \mod (s_k) .
\end{equation}
Combining \eqref{tQ2cJ9UL} and \eqref{72DGMp8S}, we get
\[
r_1c_1 + \cdots + r_dc_d \equiv r \mod (s_k) .
\]
This congruence holds for all $k$, so by the Chinese remainder theorem it also holds modulo $(s)$.
In other words,
\begin{equation}
\label{U2bQ7YxP}
r = r_1c_1 + \cdots + r_dc_d + ps
\end{equation}
for some $p \in \mathcal{O}$.
Dividing \eqref{U2bQ7YxP} by $s$ we obtain the announced partial fraction decomposition.
\end{proof}

Note that \autoref{GrJNo6wT} allows for the computation of individual numerators $r_i$ without necessity to compute the complete partial fraction decomposition.
It is also implicitly underlying Kung and Tong's fast algorithm \cite{kung1977} for partial fraction decomposition and Xin's algorithm \cite{xin2004} for MacMahon's partition analysis.
In my opinion, \autoref{GrJNo6wT} is largely underestimated and should, indeed, be taught to students as a standard method for partial fraction decomposition.

We continue by studying the important special case that $\mathcal{O} = K[w]$ is a polynomial ring and we have a partial fraction denominator of the form $1-aw^n$.
This is the kind of situation typically arising during the computation of a vector partition function.

\begin{definition}
Let $n \in \posint$.
Let $K$ be a field of characteristic not dividing $n$, containing the $n$-th roots of unity.
Let $a \in (K^\units)^n$ be an $n$-th power.
Let $f(w) \in K[w]$ be a polynomial such that $f(\alpha^{-1}) \neq 0$ if $\alpha^n = a$.
With these we associate the \emph{generalized Dedekind sum} $S_{f,a}(n) \in K$ given by 
\[
	S_{f,a}(n)
	= \frac 1n \sum_{\substack{\alpha \in K \\ \alpha^n = a}} \frac{1}{f(\alpha^{-1})}.
\]
\end{definition}

Note that $S_{f,a}(n)$ is the type of generalized Dedekind sums studied by Gessel \cite{gessel1997}.
In this situation, \autoref{GrJNo6wT} specializes as follows:

\begin{lemma}
\label{bH6NGuYt}
Let $n \in \posint$.
Let $K$ be a field of characteristic not dividing $n$, containing the $n$-th roots of unity.
Let $a \in (K^\units)^n$ be an $n$-th power.
Let $f(w) \in K[w]$ be coprime to $1-aw^n$.
Consider the partial fraction decomposition
\[
	\frac{1}{(1-aw^n)f(w)} = \frac{A(w)}{1-aw^n} + \frac{B(w)}{f(w)}
\]
with $\deg(A(w)) < n$ and $\deg(B(w)) < \deg(f(w))$.
Then the constant term of $A(w)$ is $A(0) = S_{f,a}(n)$.
\end{lemma}

\begin{proof}
We calculate the partial fraction decomposition with respect to the finer decomposition
\[
(1-aw^n)f(w) = f(w) \prod_{\alpha^n = a} (1-\alpha w)
\]
Let $A_\alpha$ denote the numerator of $1-\alpha w$ in this finer partial fraction decomposition.
As $\deg(1-\alpha w) = 1$ we have $A_\alpha \in K$.
So reduction and modular inversion in \autoref{GrJNo6wT} are just evalutation of polynomials and inversion in $K$, respectively.
Hence
\[ \begin{split}
A_\alpha
&= \Bigl( f(w) \prod_{\tilde \alpha \neq \alpha} (1-\tilde \alpha w) \Bigr\vert_{w=\alpha^{-1}} \Bigr)^{-1} \\
&= \Bigl( f(\alpha^{-1}) \prod_{\zeta^n = 1, \zeta \neq 1} (1-\zeta) \Bigr)^{-1}
= \frac{1}{nf(\alpha^{-1})} .
\end{split} \]
As $A(0) = \sum_{\alpha^n = a} A_\alpha$ the lemma follows.
\end{proof}

\autoref{bH6NGuYt} is provided here mostly as a starting point for further investigations.
In the version of the algorithm explained below we will decompose $1-aw^n$ into linear factors and treat all the poles individually.

Another important consequence of \autoref{GrJNo6wT} is that the type of partial fraction decompositions \eqref{Nb7ZAmCp} used to compute vector partition functions can be computed ``for formal $b$'' in the following sense:

\begin{proposition}
\label{RpjtYUYa}
Let $K$ be a field and $w$ a variable.
Fix $a \in K^\units$, $m \in \posint$, and $f(w) \in K[w]$ such that $f(0) \neq 0$ and $f(a^{-1}) \neq 0$.
For each $b \in \mathbf{N}$ consider the partial fraction decomposition
\[
	\frac{1}{(1-aw)^mf(w)w^b}
	= \frac{A_1(b;w)}{(1-aw)^m}
	+ \frac{A_2(b;w)}{f(w)}
	+ \frac{B(b;w)}{w^b} .
\]
Then $A_1(b;w)$ is given by an effectively computable expression in $b$.
\end{proposition}

\begin{proof}
By \autoref{GrJNo6wT}, $A_1(b;w)$ is the inverse of $f(w)w^b$ modulo $(1-aw)^m$.
Let $\tilde w$ be the local coordinate at $a^{-1}$, i.e., $w = a^{-1} + \tilde w$, and carry out all further computations in $K[\tilde w]/(\tilde w^m)$.
Then the inversion problem is reduced to invert $f(a^{-1} + \tilde w)(a^{-1} + \tilde w)^b = f(a^{-1} + \tilde w)a^{-b}(1 + a \tilde w)^b$.
Let $g(\tilde w), h(\tilde w) \in K[\tilde w]/(\tilde w^m)$ be the inverses of $f(a^{-1} + \tilde w)$ and $1 + a\tilde w$, respectively.
Then $g(\tilde w)$ can be computed explicitly.
The inverse of $f(a^{-1} + \tilde w)(a^{-1} + \tilde w)^b$ is $g(\tilde w)a^bh(\tilde w)^b$.
Note the general formula
\begin{equation}
\label{MWKpNdwa}
	\frac{1}{(1-x)^b}
	= \sum_{j=0}^\infty \binom{j+b-1}{j}x^j 
	= 1 + bx + \frac{b(b+1)}2 x^2 + \cdots .
\end{equation}
By \eqref{MWKpNdwa} we have
\[
	h(\tilde w)^b = \sum_{j=0}^{m-1} \binom{j+b-1}{j} (-a)^j\tilde w^j .
\]
Hence we obtain explicit expressions for the coefficients of $h(\tilde w)^b$, hence for $g(\tilde w)a^bh(\tilde w)^b$.
Substituting $\tilde w = w - a^{-1}$ we obtain $A_1(b;w)$.
\end{proof}

\begin{example}
Suppose we want to compute the first numerator of the partial fraction decomposition
\[
	\frac{1}{(1-w)^2w^b}
	= \frac{A(b;w)}{(1-w)^2} + \frac{B(b;w)}{w^b} .
\]
By \autoref{GrJNo6wT} we have to invert $w^b$ modulo $(1-w)^2$.
We do this as indicated in the proof of \autoref{RpjtYUYa}.
Restated in the local coordinate at $1$, we have to invert $(1 + \tilde w)^b$ modulo $\tilde w^2$.
This inverse is $(1-\tilde w)^b = 1 - b\tilde w$, a special case of \eqref{MWKpNdwa}.
Substituting $\tilde w = w - 1$ we obtain the result
\[
	A(b;w) = 1 - b(w-1) = b+1 - bw .
\]
Note that the constant term is $A(b;0) = b + 1$.
We have thus computed our first partition function, namely $\phi_{1,1}(b) = b + 1$.
\end{example}

\begin{example}
For a more complex example, consider the partial fraction decomposition
\[ \begin{split}
	\frac{1}{(1-w^2)(1-w^4)w^b}
	&= \frac{1}{(1-iw)(1+iw)(1+w)^2(1-w)^2w^b} \\
	&= \frac{A_1(b)}{1-iw}
	+ \frac{A_2(b)}{1+iw}
	+ \frac{A_3(b;w)}{(1+w)^2}
	+ \frac{A_4(b;w)}{(1-w)^2}
  + \frac{B(b;w)}{w^b} .
\end{split} \]
By direct application of \autoref{GrJNo6wT} we obtain $A_1(b) = \frac{i^b}{8}$ and $A_2(b) = \frac{i^{-b}}{8}$.
To determine $A_3(b;w)$ we proceed as in the proof of \autoref{RpjtYUYa}.
We have to invert $(1-iw)(1+iw)(1-w)^2w^b$ modulo $(1+w)^2$.
In the local coordinate $\tilde w$ at $-1$, the first factor is $1 - i (-1 + \tilde w) = (1 + i)(1 - \frac{i}{1 + i} \tilde w)$, so its inverse is $\frac{1}{1+i}(1+\frac{i}{1+i}\tilde w)$.
The second and third factor can be inverted similarly.
The crucial step is to invert $(-1 + \tilde w)^b = (-1)^b(1 - \tilde w)^b$.
This results in $(-1)^{-b}(1 + \tilde w)^b = (-1)^{-b}(1 + b \tilde w)$.
One would proceed by multiplying the four inverses modulo $\tilde w^2$ and changing the coordinate back to $w$, thus obtaining $A_3(b;w)$.
The remaining numerator $A_4(b;w)$ can be computed similarly.

To carry out the complete calculation by hand would be a considerable amount of work.
Yet I hope to have convinced you by these indications that the proof of \autoref{RpjtYUYa} does indeed furnish an algorithm to compute formal expressions for the numerators.
\end{example}

\section{Constant terms and partial fraction decomposition}
\label{vYhF777c}

We continue by making precise in which sense partial fraction decompositions can be used to compute the constant term of a rational function.
The first step is to define the rings and fields involved in the computations.
We consider the generating function $f_A(z)$ given in \eqref{Jc8LhJrg} as a function in the last variable $z_m$.
If the last row of $A$ only contains the numbers $0$ and $1$, all poles of $f_A(z)$ lie in $\mathbf{C}(z_1, \ldots, z_{m-1})$.
In the general case they are only contained in an algebraic extension.
This gives rise to the following definition.

\begin{definition}
For a field $K$, let $K((\infrt w))$ denote the field of formal series
\[
	f(w) = \sum_{k = M}^\infty c_{k/N} w^{k/N}
\]
for some $M \in \mathbf{Z}$, $N \in \posint$ with coefficients $c_{k/N} \in K$.
Similarly, let $K[\infrt w]$ denote the ring of formal finite sums $\sum_{k = M}^{M'} c_{k/N} w^{k/N}$ for some $M, M' \in \mathbf{Z}$, $N > 0$.
Let $K(\infrt{w})$ denote the quotient field of $K[\infrt{w}]$.
\end{definition}

Beware that $K((\infrt{w}))$ is not complete in the natural topology.
For example, the series $\sum_{k=2}^\infty w^{k+\frac 1k}$ does not converge.

Given a rational function $f(w) \in \mathbf{C}(w)$ in a complex variable $w$, there is a Laurent series expansion for each open annulus not containing any pole of $f(w)$.
We define the ``standard'' expansion to be the one on the annulus $\{ z \in \mathbf{C} : 0 < \lvert z \rvert < \epsilon \}$ for some $\epsilon > 0$.
If we deal with rational functions over arbitrary fields, or even functions involving rational exponents as defined above, there is no notion of convergence.
Yet we can define the standard expansion as follows:

\begin{definition}
\label{kmaa6twB}
Let $K$ be a field and $w$ a formal variable.
The \emph{standard expansion} is the embedding $K(\infrt{w}) \to K((\infrt{w}))$ over $K$ which maps $w^q$ to $w^q$ for all $q \in \mathbf{Q}$.
The \emph{constant coefficient} of $f(w) \in K(\infrt{w})$ is the coefficient of $w^0$ in the standard expansion.
It denoted by $\operatorname{const}_w f(w) \in K$.
\end{definition}

Using this definition we can state the lemma underlying the partial fraction method precisely.
With regard to the algorithmic application, we use the maximal abelian extension $\mathbf{Q}^\mathrm{ab}$ of $\mathbf{Q}$ as a ground field instead of $\mathbf{C}$.

\begin{lemma}
\label{JpeiE9iZ}
Consider a vector partition function $\phi_A$.
Suppose that $A \in \nonnegint^{(m,d)}$ has nonnegative coefficients.
Let $f_A(z) \in \mathbf{Q}^{\mathrm{ab}}(\infrt{z_1}, \ldots, \infrt{z_m})$ be the generating function of $\phi_A$ as in \eqref{Jc8LhJrg}.
Then
\[
	\phi_A(b)
	= \mathrm{const}_{z_1} \cdots \operatorname{const}_{z_m} f_A(z)z^{-b}.
\]
for all $b \in \mathbf{Q}^m$.
\end{lemma}

We will write $\const_z = \const_{z_1} \cdots \const_{z_m}$ for this type of iterated constant coefficients.
The basic principle to successfully apply the partial fraction method is to alternate working in fields of type $K(\infrt w)$ and fields of rational functions $K(w)$.
Namely, partial fraction decomposition is applied to rational functions, but produces expressions with rational exponents.
The following trivial lemma allows to clear out the rational exponents for the next iteration.

\begin{lemma}
\label{XCu4pJvt}
Let $f(w) = \sum_{q \in \mathbf{Q}} c_q w^q \in K((\infrt w))$ and $r \in \mathbf{Q}$.
Let $f(w^r) = \sum_{q \in \mathbf{Q}} c_q w^{rq}$.
Then $\operatorname{const}_w f(w) = \operatorname{const}_w f(w^r)$.
\end{lemma}

We are now in position to make the statement underlying the partial fraction method precise.
For $q \in \mathbf{Q}$, let $e(q) = e^{2\pi i q}$.
To refer to sets of variables, we use the notations $z = (z_1, \ldots, z_m)$, $z' = (z_1, \ldots, z_{m-1})$ and $w = z_m$.
Similarly, $c_k'$ denotes the vector obtained from $c_k$ by omitting the last component, so $c_k = (c_k', a_{mk})$.

\begin{proposition}
\label{gXVNAiBb}
Let $A \in \nonnegint^{(m,d)}$, $b \in \mathbf{Z}^m$, and $q \in \mathbf{Q}^d$.
Suppose that the columns $c_1, \ldots, c_d$ of $A$ are pairwise linearly independent, that $a_{mk} > 0$ for $k = 1, \ldots, d$, and that $b_m \geq 0$.
Consider $f(z) = f_{A,b,q}(z) \in \mathbf{Q}^\mathrm{ab}(z)$ of the form
\[
	f_{A,b,q}(z) = z^{-b} \prod_{k=1}^d \frac{1}{1-e(q_k)z^{c_k}} .
\]
Then one can effectively compute $c_n \in \mathbf{Q}^\mathrm{ab}$, $A^{(n)} \in \nonnegint^{(m-1,d)}$, $b^{(n)} \in \mathbf{Z}^{m-1}$, and $q^{(n)} \in \mathbf{Q}^d$ for $n = 1, \ldots, N$ such that
\[
	\const_z f(z)
	= c_1 \const_{z'} f_1(z') + \cdots + c_N \const_{z'} f_N(z')
\]
for $f_n(z') = f_{A^{(n)},b^{(n)},q^{(n)}}(z')$.
\end{proposition}

Note that among the restrictions only the one about the linear independence of the columns is severe:
If $b_m < 0$ we know immediately that $\const_z f_{A,b,q}(z) = 0$.
The factor corresponding to a possible $k$ with $a_{mk} = 0$ is contained in $\mathbf{Q}^\mathrm{ab}(z')$, hence just a scalar.

\begin{proof}
Consider $f(z) = f(z'; w)$ as a function of $w$.
All poles of $f(z'; w)$ are contained in $\mathbf{Q}^\mathrm{ab}(\infrt{z'})$.
More precisely, the roots of $1-e(q_k)z^{c_k}$ are $w = e \left( \frac{l - q_k}{a_{mk}} \right) {z'}^{{-c_k'}/{a_{mk}}}$ for $l = 0, \ldots, a_{mk}-1$.
Hence this factor decomposes as
\[
	1-e(q_k)z^{c_k} = \prod_{l=0}^{a_{mk}-1} \left( 1 - e \left( \frac{q_k - l}{a_{mk}} \right) {z'}^{{c_k'}/{a_{mk}}} w \right) .
\]
This gives rise to the partial fraction decomposition
\[
	f_{A,b,q}(z'; w)
	= \frac{B(b, z'; w)}{z^b}
	+ \sum_{k=0}^d \sum_{l=0}^{a_{mk}-1} \frac{A_{k,l}(b;z')}{1 - e \left( \frac{q_k - l}{a_{mk}} \right) {z'}^{{c_k'}/{a_{mk}}} w} .
\]
Here $\deg_w B(b, z'; w) < b_m$, so the corresponding summand does not contribute to $\const_w f_{A,b,q}(z'; w)$.
As the remaining poles are simple, the numerators $A_{k,l}(b;z')$ do not depend on $w$.
Hence
\[
	\const_w f_{A,b,q}(z';w) = \sum_{k=0}^d \sum_{l=0}^{a_{mk}-1} A_{k,l}(b;z') .
\]
It follows that
\[
	\const_z f_{A,b,q}(z';w)
	= \const_{z'} \const_w f_{A,b,q}(z';w)
	= \sum_{k=0}^d \sum_{l=0}^{a_{mk}-1} \const_{z'} A_{k,l}(b;z') .
\]
Our strategy to construct $f_1(z'), \ldots, f_N(z')$ is hence to transform the summands $A_{k,l}(b;z')$ suitably.
By \autoref{GrJNo6wT} they are given by
\begin{multline*}
	A_{k,l}(b; z')
	={1} \Bigg/ 
	\prod_{\tilde k \neq k} (1-e(q_{\tilde k})z^{a_{\tilde k}})
	\cdot z^b \\
	\cdot \prod_{\tilde l \neq l} \left( 1 - e \left( \frac{q_k - \tilde l}{a_{mk}} \right) z^{{a_k}/{a_{mk}}} \right)
	\Bigg\vert_{w = e \left( \frac{l - q_k}{a_{mk}} \right) {z'}^{-{c_k'}/{a_{mk}}}} .
\end{multline*}
After the substitution, the first factor in the denominator becomes
\[
\prod_{\tilde k \neq k} \left( 1-e \left( q_{\tilde k}+\frac{a_{m\tilde k(l-q_k)}}{a_{mk}} \right) {z'}^{\frac{a_{mk}c_{\tilde k}' - a_{m\tilde k}c_k'}{a_{mk}}} \right) ,
\]
the second factor becomes
\[
	e \left( \frac{b_m(l-q_k)}{a_{mk}} \right)
{z'}^{\frac{a_{mk}b' - b_mc_{k}'}{a_{mk}}} ,
\]
and the last factor becomes
\[
	\prod_{{\tilde l} \neq l} \left( 1 - e \left( \frac{l - {\tilde l}}{a_{mk}} \right) \right) \\
	= a_{mk} .
\]
We want to construct $c_1, \ldots, c_N$ and $f_1(z'), \ldots, f_N(z')$ from $A_{k,l}(b,z')$.
To simplify notation, change the indexing to $n = (k,l)$.
The idea is to choose
\[
	c_{k,l} = \frac 1 {e \left( \frac{b_m(l-q_k)}{a_{mk}} \right) a_{mk}}
\]
and
\[ \begin{split}
	f_{k,l}(z')
	&= \frac{1}{c_{k,l}} A_{k,l}(b; z') \\
	&= 	{z'}^{\frac{b_mc_{k}' - a_{mk}b'}{a_{mk}}}
\prod_{\tilde k \neq k} \left( 1-e \left( q_{\tilde k}+\frac{a_{m\tilde k(l-q_k)}}{a_{mk}} \right) {z'}^{\frac{a_{mk}c_{\tilde k}' - a_{m\tilde k}c_k'}{a_{mk}}} \right)^{-1} .
\end{split} \]
This is not quite possible, because $b^{(k,l)} = (b_mc_{k}' - a_{mk}b')/a_{mk}$ contains rational numbers as opposed to integers, and the columns $(a_{mk}c_{\tilde k}' - a_{m\tilde k}c_k')/a_{mk}$ of $A^{(k,l)}$ contain rational numbers as opposed to nonnegative integers.
By a multivariate version of \autoref{XCu4pJvt} we can substitute $z_j$ by $z_j^{a_{mk}}$ without changing the constant coefficient, thereby obtaining integers in both places.
To to make $A^{(k,l)}$ nonnegative, use the formula
\begin{equation}
\label{h9RGNqNp}
	\frac{1}{1-z_j^{-b}} = - \frac{z_j^b}{1-z_j^b}
\end{equation}
as necessary.
The sign will amount to a change of $c_{k,l}$, the numerator to a change of $b^{(k,l)}$.
\end{proof}

The most important feature of \autoref{gXVNAiBb} is that the algorithm does not branch on a condition depending on $b$.
Hence, although stated for specific $b \in \mathbf{Z}^n$, it can actually compute formal expressions in $b$.

\begin{example}
Let $A = \bigl( \begin{smallmatrix} 1 & 1 \\ 3 & 1 \end{smallmatrix} \bigr)$.
Then $\phi_A(a, b) = \operatorname{const} f_A(z,w)z^{-a}w^{-b}$, and
\[
	f_A(z,w)z^{-a}w^{-b} = \frac{1}{(1 - zw^3) (1 - zw) z^a w^b} .
\]
This is a trivial example, included here only for the sake of illustration.
In the first step of the iteration we compute the constant term with respect to $w$.
It $b < 0$ then $\phi_A(a,b) = 0$, so suppose $b \geq 0$.
As $z^{-a}$ is a scalar now, we start with the following partial fraction decomposition specializing \eqref{Nb7ZAmCp}:
\[
	\frac{1}{(1 - zw^3) (1 - zw) w^b}
	= \frac{A_1(b;w)}{1-zw^3}
	+ \frac{A_2(b;w)}{1-zw}
	+ \frac{B(b;w)}{w^b}
\]
In order to apply \autoref{bH6NGuYt} we have to ensure that $z$ is a third power.
So let $z^{1/3}$ be a third root of $z$, and use $K = \mathbf{C}(z^{1/3})$ as the field of scalars.
Then
\[ \begin{split}
	\operatorname{const}_w \frac 1{(1 - zw^3) (1 - zw) w^b}
	&= A_1(b;0) + A_2(b;0) \\
	&= S_{(1-zw)w^b,z}(3) + S_{(1-zw^3)w^b,z}(1) .
\end{split} \]
Let $\zeta$ be a primitive third root of unity.
Then
\begin{align*}
	S_{(1-zw)w^b,z}(3)
	&= \frac 13 \Bigg( \frac 1{(1-z^{2/3})z^{-b/3}}
	+ \frac 1{(1-\zeta z^{2/3})\zeta^bz^{-b/3}} \\
	& \quad + \frac 1{(1-\zeta^2z^{2/3})\zeta^{2b}z^{-b/3}} \Bigg) , \\
	S_{(1-zw^3)w^b,z}(1)
	&= \frac 1{(1-z^{-2})z^{-b}}
	= -\frac 1{(1-z^2)z^{-b-2}} .
\end{align*}
Hence $\phi_A(a,b)$ is the constant coefficient of the univariate function
\begin{multline*}
	\frac 13 \Bigg(
	\frac 1{(1-z^{2/3})z^{a-b/3}}
+ \frac 1{(1-\zeta z^{2/3})\zeta ^bz^{a-b/3}} \\
+ \frac 1{(1-\zeta ^2z^{2/3})\zeta ^{2b}z^{a-b/3}} \Bigg)
- \frac 1{(1-z^2)z^{a-b-2}} .
\end{multline*}
In order to eliminate the fractional powers, we substitute $z$ by $z^3$ according to \autoref{XCu4pJvt} in the first three summands and obtain
\[ \begin{split}
	\phi_A(a,b)
	&= \const_z
	\frac 13 \Bigg(
	\frac 1{(1-z^{2})z^{3a-b}}
	+ \frac 1{(1-\zeta z^{2})\zeta^bz^{3a-b}} \\
	& \quad + \frac 1{(1-\zeta ^2z^{2})\zeta ^{2b}z^{3a-b}} \Bigg)
  - \frac 1{(1-z^2)z^{a-b-2}} .
\end{split} \]
From this we would compute a quasi-polynomial expression for $\phi_A(a,b)$ as a second step.
\end{example}

\begin{example}
Let $A = \left( \begin{smallmatrix} 1&0&1 \\ 0&1&1 \end{smallmatrix} \right)$.
Then $\phi_A$ is Kostant's partition function associated with a root system of type $A_2$.
A prototype implementation of the partial fraction method produces, given $A$, the following output:
\[
	\phi_A(a,b)
	= \left\{ \begin{array}{cl}
	\left\{ \begin{array}{cl}
	a+1 & \text{if } a \geq 0 \\
	0 & \text{otherwise}
	\end{array} \right\}
	-
	\left\{ \begin{array}{cl}
	a-b & \text{if } a-b \geq 1 \\
	0 & \text{otherwise}
	\end{array} \right\}
	& \text{if } b \geq 0 \\
	0 & \text{otherwise.}
	\end{array} \right.
\]
Note that the traces of the algorithm are still visible in the output:
The output consists of nested case statements, where the second alternative is always $0$.
This corresponds to to the fact that if $b_m < 0$ then $\const_z f(z)z^{-b} = 0$.
The outer case statement corresponds to the first step, the partial fraction decomposition with respect to $z_2$.
Hence the condition is simply $b \geq 0$.
As there are two nonzero components in the second row of $A$, the first iteration yiels a sum of two rational functions in $z_1$.
These correspond to the two terms of the difference.
Note that for the second term there is a change of sign, and the condition is that a certain function of $a$ and $b$ must be positive.
This is because of an application of \eqref{h9RGNqNp} in the first step, as explained in the end of the proof of \autoref{gXVNAiBb}.
\end{example}

\section*{Acknowledgements}

I wish to thank Matthias Beck and the Department of Mathematics at the San Francisco State University for their hospitality during the preparation of this extended abstract.

\bibliographystyle{amsplain}
\bibliography{/Users/thomas/Documents/Bibliothek/bibliothek}

\providecommand{\doi}[1]{\href{http://dx.doi.org/#1}{\nolinkurl{doi:#1}}}
  \providecommand{\arxiv}[1]{\href{http://arxiv.org/abs/#1}{\nolinkurl{arXiv:#%
1}}}
  \providecommand{\mr}[1]{\href{http://www.ams.org/mathscinet-getitem?mr=MR#1}%
{MR#1}}
  \providecommand{\urn}[1]{\href{http://nbn-resolving.de/urn:#1}{\nolinkurl{ur%
n:#1}}}
  \providecommand{\zbl}[1]{\href{http://www.zentralblatt-math.org/zmath/en/adv%
anced/?q=an:#1&format=complete}{Zbl \nolinkurl{#1}}} \selectlanguage{USenglish}
\providecommand{\bysame}{\leavevmode\hbox to3em{\hrulefill}\thinspace}
\providecommand{\MR}{\relax\ifhmode\unskip\space\fi MR }
\providecommand{\MRhref}[2]{%
  \href{http://www.ams.org/mathscinet-getitem?mr=#1}{#2}
}
\providecommand{\href}[2]{#2}
\begin{thebibliography}{10}

\bibitem{baldoni2006}
M.~Welleda Baldoni, Matthias Beck, {\relax Ch}arles Cochet, and Mich{\`e}le
  Vergne, \emph{Volume computation for polytopes and partition functions for
  classical root systems}, Discrete Comput. Geom. \textbf{35} (2006), 551--595,
  \doi{10.1007/s00454-006-1234-2}.

\bibitem{baldoni2004}
W.~Baldoni-Silva, J.~A. {D}e Loera, and M.~Vergne, \emph{Counting integer flows
  in networks}, Found. Comput. Math. \textbf{4} (2004), 277--314,
  \doi{10.1007/s10208-003-0088-8}.

\bibitem{barvinok1994}
Alexander~I. Barvinok, \emph{A polynomial time algorithm for counting integral
  points in polyhedra when the dimension is fixed}, Math. Oper. Res.
  \textbf{19} (1994), 769--779, \doi{10.1287/moor.19.4.769}.

\bibitem{beck2004}
Matthias Beck, \emph{The partial-fractions method for counting solutions to
  integral linear systems}, Discrete Comput. Geom. \textbf{32} (2004),
  437--446, \doi{10.1007/s00454-004-1131-5}.

\bibitem{billey2004}
Sara Billey, Victor Guillemin, and Etienne Rassart, \emph{A vector partition
  function for the multiplicities of $\mathfrak{sl}_k \mathbf{C}$}, J. Algebra
  \textbf{278} (2004), 251--293, \doi{10.1016/j.jalgebra.2003.12.005}.

\bibitem{blakley1964}
G.~R. Blakley, \emph{Combinatorial remarks on partitions of a multipartite
  number}, Duke Math. J. \textbf{31} (1964), 335--340, errata 718,
  \doi{10.1215/S0012-7094-64-03132-1}.

\bibitem{bliem2010}
{\relax Th}omas Bliem, \emph{Chopped and sliced cones and representations of
  {K}ac-{M}oody algebras}, J. Pure Appl. Algebra (2009), in press,
  \doi{10.1016/j.jpaa.2009.10.002}.

\bibitem{boysal2008}
Arzu Boysal and Mich{\`e}le Vergne, \emph{Paradan's wall crossing formula for
  partition functions and {K}hovanski-{P}ukhlikov differential operator},
  \arxiv{0803.2810v1}, 2008.

\bibitem{dahmen1988}
Wolfgang Dahmen and {\relax Ch}arles~A. Micchelli, \emph{The number of
  solutions to linear {D}iophantine equations and multivariate splines}, Trans.
  Am. Math. Soc. \textbf{308} (1988), 509--532, \doi{10.2307/2001089}.

\bibitem{ehrhart1962}
Eug{\`e}ne Ehrhart, \emph{\foreignlanguage{french}{Sur les poly\`edres
  rationnels homoth\'etiques \`a {$n$}\ dimensions}}, C. R. Acad. Sci. Paris
  \textbf{254} (1962), 616--618.

\bibitem{gessel1997}
Ira~M. Gessel, \emph{Generating functions and generalized {D}edekind sums},
  Electron. J. Comb. \textbf{4} (1997), no.~2,
  \href{http://www.combinatorics.org/Volume_4/Abstracts/v4i2r11.html}{R11}.

\bibitem{kostant1959}
Bertram Kostant, \emph{A formula for the multiplicity of a weight}, Trans. Am.
  Math. Soc. \textbf{93} (1959), 53--73, \doi{10.2307/1993422}.

\bibitem{kung1977}
H.~T. Kung and D.~M. Tong, \emph{Fast algorithms for partial fraction
  decomposition}, SIAM J. Comput. \textbf{6} (1977), 582--593,
  \doi{10.1137/0206042}.

\bibitem{milev2009}
Todor Milev, \emph{Partial fraction decompositions and an algorithm for
  computing the vector partition function}, \arxiv{0910.4675v2}, 2009.

\bibitem{sturmfels1995}
Bernd Sturmfels, \emph{On vector partition functions}, J. Comb. Theory, Ser. A
  \textbf{72} (1995), 302--309, \doi{10.1016/0097-3165(95)90067-5}.

\bibitem{szenes2003}
Andr{\'a}s Szenes and Mich{\`e}le Vergne, \emph{Residue formulae for vector
  partitions and {E}uler-{M}ac{L}aurin sums}, Adv. Appl. Math. \textbf{30}
  (2003), 295--342, \doi{10.1016/S0196-8858(02)00538-9}.

\bibitem{verdoolaege2007}
Sven Verdoolaege, Rachid Seghir, Kristof Beyls, Vincent Loechner, and Maurice
  Bruynooghe, \emph{Counting integer points in parametric polytopes using
  {B}arvinok's rational functions}, Algorithmica \textbf{48} (2007), 37--66,
  \doi{10.1007/s00453-006-1231-0}.

\bibitem{xin2004}
Guoce Xin, \emph{A fast algorithm for {M}ac{M}ahon's partition analysis},
  Electron. J. Comb. \textbf{11} (2004), no.~1,
  \href{http://www.combinatorics.org/Volume_11/Abstracts/v11i1r58.html}{R58}.

\end{thebibliography}

\end{document}